\documentclass[11pt]{article}

\usepackage{amsmath,amssymb,amsfonts,amsthm}
\usepackage{mathtools}
\usepackage{enumitem}
\usepackage{geometry}
\usepackage{hyperref}
\newtheorem{theorem}{Theorem}[section]
\newtheorem{proposition}[theorem]{Proposition}

\newtheorem{corollary}[theorem]{Corollary}
\theoremstyle{definition}
\newtheorem{definition}[theorem]{Definition}
\newtheorem{remark}[theorem]{Remark}

\newcommand{\F}{\mathbb{F}}
\newcommand{\Z}{\mathbb{Z}}
\newcommand{\C}{\mathbb{C}}

\newcommand{\End}{\mathrm{End}}
\newcommand{\id}{\mathrm{id}}
\newcommand{\Fun}{\mathrm{Fun}}
\newcommand{\Span}{\operatorname{Span}}

\title{Algebraic Phase Theory III: Structural Quantum Codes over Frobenius Rings}
\author{
Joe Gildea\\
Department of Computing Science and Mathematics,\\
School of Informatics and Creative Arts,\\
Dundalk Institute of Technology\\
\texttt{gildeajoe@gmail.com}}
\date{}

\begin{document}
\maketitle

\begin{abstract}
We develop the quantum component of Algebraic Phase Theory by showing that
quantum phase, Weyl noncommutativity, and stabiliser codes arise as
unavoidable algebraic consequences of Frobenius duality.
Working over finite commutative Frobenius rings, we extract nondegenerate
phase pairings, Weyl operator algebras, and quantum stabiliser codes directly
from admissible phase data, without assuming Hilbert spaces, analytic inner
products, or an externally imposed symplectic structure.

Within this framework, quantum state spaces appear as minimal carriers of
faithful phase action, and stabiliser codes are identified canonically with
self-orthogonal submodules under the Frobenius phase pairing.
CSS-type constructions arise only as a special splitting case, while general
Frobenius rings admit intrinsically non-CSS stabilisers.
Nilpotent and torsion structure in the base ring give rise to algebraically
protected quantum layers that are invisible to admissible Weyl-type errors.

These results place quantum stabiliser theory within Algebraic Phase Theory:
quantisation emerges as algebraic phase induction rather than analytic
completion, and quantum structure is information-complete at the level of
algebraic phase relations alone.
Throughout, we work over finite Frobenius rings, which are precisely the
base rings for which admissible phase data become strongly admissible, and
in this regime the full quantum formalism is forced by Frobenius duality.
\end{abstract}

\medskip
\noindent\textbf{Mathematics Subject Classification (2020).}
Primary 81P70; Secondary 13M10, 20C15, 81R05, 94B05.

\medskip
\noindent\textbf{Keywords.}
Algebraic Phase Theory, Frobenius rings, Weyl algebras, stabiliser codes,
quantum error correction, non CSS quantum codes, nilpotent protection,
algebraic quantisation, phase duality.

\section{Introduction}

This paper develops the quantum component of Algebraic Phase Theory (APT) by
showing that quantum phase, Weyl noncommutativity, and stabiliser codes arise
as unavoidable algebraic consequences of Frobenius duality.  No analytic inner
product, Hilbert space structure, or externally imposed symplectic form is
assumed, and instead quantisation emerges intrinsically from admissible
algebraic phase data.  In the first two papers of this series, APT was
developed in an abstract setting, establishing defect, filtration, rigidity,
and boundary phenomena as structural invariants of phase interaction.  The
present work applies this framework to the finite Frobenius setting and shows
that the basic structures of quantum information theory are already forced at
this purely algebraic level.

From the perspective of Algebraic Phase Theory, finite Frobenius rings are
precisely those base rings for which admissible phase data become strongly
admissible.  In this regime, defect propagation, canonical filtration, and
finite termination are not auxiliary assumptions but intrinsic features forced
by the axioms.  All quantum constructions carried out in this paper therefore
occur within the rigid core of APT, with no need for analytic or topological
input.

Viewed from the perspective of existing literature, the structures isolated
here intersect several established traditions.  Stabiliser codes and quantum
error correction originate in the work of Calderbank, Shor, and Steane
\cite{CalderbankShor,Steane}, with their algebraic formulation clarified by
Gottesman \cite{Gottesman} and extended to nonbinary settings by Rains and
others \cite{Rains,Ketkar}.  These approaches rely fundamentally on Hilbert
spaces equipped with symplectic forms.  Independently, Weyl and Heisenberg
operator algebras have long played a central role in representation theory and
harmonic analysis, see for example
\cite{vonNeumann1931,Weil1964SurCertains,Howe1979Heisenberg,Folland,Mackey1958Induced},
with modern refinements such as the Weil representation in finite and
characteristic two settings \cite{GurevichHadani}.  In parallel, Frobenius
rings and their duality theory have been developed extensively in algebra and
coding theory
\cite{AndersonFuller1992Rings,Jacobson1956Structure,Lam2001FirstCourse,Honold2001QuasiFrobenius,Wood1999Duality,GreferathSchmidt2000Characters},
particularly in connection with MacWilliams equivalence and classical coding
theory over rings \cite{DinhLopezPermouth}.  What has been missing is a
framework in which these components are not merely compatible but structurally
forced by a single algebraic principle.  In standard treatments, symplectic
forms, commutation laws, and quantum compatibility conditions are imposed
externally.  In contrast, we show that Weyl commutation relations, stabiliser
orthogonality, and error protection mechanisms follow automatically from
Frobenius duality and functorial phase interaction.

A central theme of this paper is the algebraic forcing of quantum structure.
For any finite commutative Frobenius ring together with a generating
character, and for any admissible phase datum taking values in that ring, the
core elements of quantum theory arise canonically and functorially from the
interaction laws encoded by the algebraic phase.  The Frobenius pairing
extracted from the phase data determines a character-valued interaction law,
and this interaction law forces the appearance of Weyl operators whose
noncommutativity is governed by a bicharacter emerging directly from the
Frobenius pairing.  Once these operators appear, their commutation behaviour
identifies stabiliser subgroups precisely with those submodules that are
self-orthogonal under the Frobenius pairing.  Moreover, any nilpotent or
torsion structure in the base ring gives rise to intrinsically protected
layers of quantum information, since such elements render certain phase
interactions algebraically invisible at the operator level.  In this way, the
full stabiliser formalism, including noncommutativity, orthogonality, and
error protection phenomena, emerges as an algebraic invariant of Frobenius
phase geometry rather than as an imposed analytic structure.

The results of this paper situate quantum stabiliser theory firmly within the
APT paradigm.  Quantisation appears as algebraic phase induction rather than
analytic completion, and quantum structure is shown to be completely encoded
at the level of phase relations.  The Frobenius hypothesis plays a decisive
role.  It identifies the exact structural boundary within APT at which quantum
rigidity persists.  Inside this boundary, the Weyl calculus, the
noncommutative Heisenberg action, and stabiliser orthogonality are enforced by
the axioms.  Outside this boundary, generating characters fail to exist,
defect becomes unbounded, canonical filtration ceases to terminate, and
rigidity necessarily breaks down.

This paper forms part of a broader programme to develop Algebraic Phase Theory
as a unified approach to phase geometry, representation theory, and quantum
structure.  Together with the companion papers
\cite{GildeaAPT1,GildeaAPT2,GildeaAPT4,GildeaAPT5,GildeaAPT6}, it contributes
to a complete foundation for APT.  In this way the algebraic mechanisms
governing coherent phase interaction are shown to control both classical and
quantum behaviour, and the emergence of quantum structure is explained
entirely within the algebraic domain.

\section{Structural Background from Algebraic Phase Theory}

Before specializing to Frobenius rings or quantum systems, we record the
minimal structural input required to speak meaningfully about phase and its
interaction.  This section provides the representation independent framework
from which all subsequent algebraic structures are extracted.  No analytic,
topological, or Hilbert space structure is assumed at any point.

Rather than starting from an algebraic object with a prescribed operator
algebra, we single out only the data needed to discuss phases, their behaviour
under maps, and the interaction laws that govern them.  All algebraic structure
used later in the paper is extracted canonically from this data and is never
imposed by hand.

Our use of additive characters and phase operators relies on standard results
from character theory and the representation theory of finite groups; see, for
example, \cite{Isaacs,Serre,CurtisReiner}.

An admissible phase framework is built from three ingredients: an underlying
additive object, a chosen collection of phase functions defined on it, and a
distinguished interaction law governing the behaviour of the corresponding
phase operators.  At this stage no commitment is made as to whether phases are
realised as functions, operators, or equivalence classes.  The point of this
separation is that all algebraic structure appearing later arises by necessity
rather than assumption.

\begin{remark}
In the finite Frobenius setting considered throughout this paper, admissible
phase data is automatically strongly admissible.  All subsequent constructions
therefore occur in the maximally rigid regime of Algebraic Phase Theory.  Phase
data exhibiting weaker propagation or boundary slippage lies outside the scope
of this work.
\end{remark}

  \subsubsection*{Phase operators and functoriality}

The first structural ingredient needed for APT is the passage from abstract
phase functions to algebraically visible operators.

Let $R$ be a finite ring equipped with an additive character
$\chi:R\to\mathbb C^\times$.  For any additive object $A$, each
phase $\phi:A\to R$ induces a phase multiplication operator on
$\Fun(A,\mathbb C)$ given by
\[
(M_\phi f)(x):=\chi(\phi(x))\,f(x).
\]
This operator realization makes abstract phase data algebraically visible
without assuming analytic or topological structure.

For a morphism $f:A\to A'$, phases pull back functorially by
$f^\ast(\phi'):=\phi'\circ f$.  Functoriality is essential for extracting
canonical algebraic structure.

\subsubsection*{Admissible phase data}

We now isolate the minimal formal requirements for a well behaved phase
theory.

\begin{definition}
A \emph{weakly admissible phase datum} consists of a triple $(A,\Phi,\circ)$ such
that:
\begin{enumerate}[label=(W\arabic*)]
\item $A$ is an object of $\mathsf{C}$ and $\Phi(A)$ is a functorial family of
phases $\phi:A\to R$;
\item all phases have uniformly bounded additive degree;
\item the interaction law $\circ$ is specified abstractly and admits functorial
detection of defect and failure of closure at finite depth.
\end{enumerate}
No operator realization, character theory, or analytic structure is assumed.
\end{definition}

\begin{definition}
A \emph{strongly admissible phase datum} is a weakly admissible phase datum
together with a faithful operator realization of phases that detects defect,
filtration, and rigidity functorially.
\end{definition}

\begin{remark}
For finite rings, the existence of a faithful phase operator realization forces
the base ring to be Frobenius.  Thus Frobenius duality is not an additional
assumption but a structural consequence of admissibility.
\end{remark}

The purpose of admissibility is not to impose an algebraic structure but to
guarantee that such a structure can later be extracted canonically.  The
interaction law, the functoriality of phases, and the uniform structural
bound jointly ensure that the resulting algebraic phase model is intrinsic to
the phase data itself.

From an admissible phase datum one extracts an algebraic object encoding phase
interaction:
\[
(\mathcal P,\circ),
\]
where $\mathcal P$ is the algebra generated by the phase and shift operators
associated with $\Phi(A)$, and $\circ$ is the distinguished interaction law.
No additional structure on $\mathcal P$ is imposed beyond what is forced by
the admissible datum.

\subsection{Algebraic Phase Theory}

Algebraic Phase Theory (APT) formalizes the principle that admissible phase
data canonically determines an algebraic phase model $(\mathcal P,\circ)$,
unique up to intrinsic equivalence, together with structural notions such as
rigidity, stratification, complexity, and canonical filtration.

In admissible finite settings these invariants admit finite depth, and the
associated filtrations terminate canonically.  These properties form the
structural skeleton of APT and govern the behaviour of all algebraic phases
considered in this paper.

The complete axiomatic development of APT, including defect, complexity
measures, filtration theory, and termination phenomena, appears in
\cite{GildeaAPT1,GildeaAPT2}.  None of the technical machinery of defect theory
is required here; we rely only on the existence, canonicity, and functoriality
of the extracted algebraic phase model.

\section{Frobenius Phase Data and Algebraic Quantum Structure}

We now specialize the general phase–extraction framework to the algebraic
setting relevant for quantum phenomena.  Throughout the paper, admissible
phase data arises from Frobenius duality, and all subsequent quantum
structures are extracted from this duality rather than imposed externally.

A finite commutative ring $R$ is called \emph{Frobenius} if it admits a
generating character
\[
\varepsilon:(R,+)\to\C^\times
\]
such that
\[
x\longmapsto\bigl(y\longmapsto\varepsilon(xy)\bigr)
\]
is an isomorphism $R\cong\widehat{R}$ of left $R$-modules.

\begin{remark}
The existence of a generating character is precisely the condition required
for a faithful operator realisation of phase multiplication.  Thus Frobenius
duality is not an additional hypothesis: it is forced by admissibility of the
phase datum.
\end{remark}

Given such a ring, the associated Frobenius phase pairing is
\[
\langle x,y\rangle_R := \varepsilon(xy),
\]
which is biadditive, symmetric, and nondegenerate.  Frobenius duality and
its operator theoretic manifestations form the algebraic basis for the
quantum structures introduced below; see
\cite{AndersonFuller1992Rings,Jacobson1956Structure,Lam2001FirstCourse,
Honold2001QuasiFrobenius,Wood1999Duality,GreferathSchmidt2000Characters}.

\subsection{Algebraic Quantum State Spaces}

We now explain how quantum state spaces arise canonically within Algebraic
Phase Theory.  Instead of postulating Hilbert spaces, we extract the minimal
algebraic carriers required for faithful phase action.

Let $V$ be a finite free $R$-module of rank $k$, interpreted as a single
algebraic quantum site, and fix a perfect $R$-bilinear form
\[
\beta:V\times V\to R .
\]
This determines linear phase functions
\[
\phi_b(x):=\beta(b,x),
\qquad
b\in V,
\]
and hence operator phase multipliers via the character $\varepsilon$.

\begin{definition}
The induced quantum phase pairing on $V$ is
\[
\langle v,w\rangle_V := \varepsilon\!\big(\beta(v,w)\big).
\]
\end{definition}

This pairing is biadditive, symmetric, and nondegenerate.

\subsubsection{Tensor extension of the phase pairing}

For $n\ge 1$, define the $n$-site phase label space by
\[
H_n := V^{\otimes_R n}.
\]
For pure tensors $v=v_1\otimes\cdots\otimes v_n$ and
$w=w_1\otimes\cdots\otimes w_n$ in $H_n$ set
\[
\langle v,w\rangle_{H_n}
=
\prod_{i=1}^n \langle v_i,w_i\rangle_V
=
\varepsilon\!\left(\sum_{i=1}^n \beta(v_i,w_i)\right).
\]

\begin{proposition}
The induced pairing on $H_n$ is biadditive, symmetric, monoidal with respect
to tensor product, and nondegenerate.
\end{proposition}

\begin{proof}
Recall that $H_n:=V^{\otimes_R n}$, where $V$ is a finite free $R$ module
equipped with a perfect $R$-bilinear form
\[
\beta:V\times V\to R,
\]
and that the Frobenius phase pairing on $V$ is given by
\[
\langle v,w\rangle_V := \varepsilon(\beta(v,w)).
\]
For pure tensors $v=v_1\otimes\cdots\otimes v_n$ and
$w=w_1\otimes\cdots\otimes w_n$ in $H_n$, the induced pairing is defined by
\[
\langle v,w\rangle_{H_n}
:= \prod_{i=1}^n \langle v_i,w_i\rangle_V
=
\varepsilon\!\left(\sum_{i=1}^n \beta(v_i,w_i)\right).
\]

\smallskip
\noindent\emph{Biadditivity.}
Fix $w\in H_n$.  It suffices to check additivity in the first variable on pure
tensors, as such tensors generate $H_n$ additively.
Let
\[
v=(v_1,\dots,v_n), \qquad v'=(v_1',\dots,v_n')
\]
be pure tensors.  Using additivity of $\beta$ in each variable and that
$\varepsilon:(R,+)\to\C^\times$ is a group homomorphism, we compute
\[
\langle v+v',w\rangle_{H_n}
=
\varepsilon\!\left(\sum_{i=1}^n \beta(v_i+v_i',w_i)\right)
=
\varepsilon\!\left(\sum_{i=1}^n \beta(v_i,w_i)\right)
\varepsilon\!\left(\sum_{i=1}^n \beta(v_i',w_i)\right)
=
\langle v,w\rangle_{H_n}\,\langle v',w\rangle_{H_n}.
\]
Additivity in the second variable is analogous.  Hence the pairing is
biadditive.

\smallskip
\noindent\emph{Symmetry.}
Since $R$ is commutative and $\beta$ is symmetric, we have
$\beta(v_i,w_i)=\beta(w_i,v_i)$ for all $i$, and therefore
\[
\langle v,w\rangle_{H_n}
=
\varepsilon\!\left(\sum_{i=1}^n \beta(v_i,w_i)\right)
=
\varepsilon\!\left(\sum_{i=1}^n \beta(w_i,v_i)\right)
=
\langle w,v\rangle_{H_n}.
\]

\smallskip
\noindent\emph{Monoidality.}
Let $v=v^{(1)}\otimes v^{(2)}\in H_m\otimes_R H_n$ and
$w=w^{(1)}\otimes w^{(2)}\in H_m\otimes_R H_n$.
By definition of the pairing and the tensor product decomposition, we have
\[
\langle v,w\rangle_{H_{m+n}}
=
\langle v^{(1)},w^{(1)}\rangle_{H_m}\,
\langle v^{(2)},w^{(2)}\rangle_{H_n},
\]
showing that the pairing is monoidal with respect to tensor product.

\smallskip
\noindent\emph{Nondegeneracy.}
Suppose $0\neq v\in H_n$.  Write $v$ as a finite sum of pure tensors and select
a pure tensor component
$v_1\otimes\cdots\otimes v_n$ appearing with nonzero coefficient.
Since $\beta$ is perfect, for each $i$ there exists $w_i\in V$ such that
$\beta(v_i,w_i)\neq 0$.
Set $w:=w_1\otimes\cdots\otimes w_n\in H_n$.
Then
\[
\sum_{i=1}^n \beta(v_i,w_i)\neq 0
\]
in $R$, and since $\varepsilon$ is a generating character, it follows that
\[
\langle v,w\rangle_{H_n}=\varepsilon\!\left(\sum_{i=1}^n \beta(v_i,w_i)\right)\neq 1.
\]
Hence $v\notin H_n^\perp$.  This proves that the pairing is nondegenerate.

\smallskip
Therefore the induced pairing on $H_n$ is biadditive, symmetric, monoidal, and
nondegenerate.
\end{proof}

\begin{remark}
The Frobenius structure of $R$ induces a canonical character-valued phase
pairing on every tensor power.  Composite systems therefore inherit an
information-complete phase structure compatible with Weyl commutation and
stabiliser methods, without any analytic or probabilistic assumptions.
\end{remark}

\subsection*{Interpretation of the parameters $k$ and $n$}

The rank $k$ controls the internal algebraic complexity of a single quantum
site; $n$ measures the global system size.  For $R$ a finite field,
$k=1$ yields the usual qudit of dimension $|R|$.  Over general Frobenius
rings, $k$ also captures nilpotent and non-semisimple contributions to local
phase structure.

\subsection{Weyl Operators and the Algebraic Phase Model}

We now show that admissible Frobenius phase data canonically induces Weyl
commutation relations, yielding a noncommutative algebraic phase model.
Let $H_n := V^{\otimes_R n}$ denote the $n$-site phase label space.
For $a,b\in H_n$, define shift and phase operators on $\Fun(H_n,\C)$ by
\[
(T_a f)(x) := f(x-a),
\qquad
(M_b f)(x) := \varepsilon(\beta(b,x))\,f(x),
\]
where $\beta_{H_n}$ is the bilinear form on $H_n$ induced monoidally from
$\beta$.

\begin{proposition}
\label{prop:weyl-commutation}
Let $(R,\varepsilon,V,\beta)$ be Frobenius quantum data and let
$H_n=V^{\otimes_R n}$.
Then for all $a,b\in H_n$ the operators $T_a$ and $M_b$ satisfy
\[
T_a M_b=\varepsilon(\beta(b,a))^{-1}\,M_b T_a .
\]
\end{proposition}

\begin{proof}
Fix $f\in \Fun(H_n,\C)$ and $x\in H_n$.  We compute both compositions on $f$
evaluated at $x$. First,
\[
\bigl(T_a M_b f\bigr)(x)
=(M_b f)(x-a)
=\varepsilon\!\big(\beta(b,x-a)\big)\,f(x-a).
\]
Since $\beta$ is $R$-bilinear (hence additive in the second variable), we have 
\[
\beta(b,x-a)=\beta(b,x)+\beta(b,-a)=\beta(b,x)-\beta(b,a).
\]
Therefore
\[
\bigl(T_a M_b f\bigr)(x)
=\varepsilon\!\big(\beta(b,x)-\beta(b,a)\big)\,f(x-a).
\]
Using that $\varepsilon:(R,+)\to\C^\times$ is an additive character, hence a
group homomorphism, we obtain
\[
\varepsilon\!\big(\beta(b,x)-\beta(b,a)\big)
=\varepsilon\!\big(\beta(b,x)\big)\,\varepsilon\!\big(-\beta(b,a)\big)
=\varepsilon\!\big(\beta(b,x)\big)\,\varepsilon\!\big(\beta(b,a)\big)^{-1}.
\]
Thus
\[
\bigl(T_a M_b f\bigr)(x)
=\varepsilon\!\big(\beta(b,a)\big)^{-1}\,\varepsilon\!\big(\beta(b,x)\big)\,f(x-a).
\]

On the other hand,
\[
\bigl(M_b T_a f\bigr)(x)
=\varepsilon\!\big(\beta(b,x)\big)\,(T_a f)(x)
=\varepsilon\!\big(\beta(b,x)\big)\,f(x-a).
\]
Comparing the two expressions gives, for all $f$ and $x$,
\[
\bigl(T_a M_b f\bigr)(x)
=\varepsilon\!\big(\beta(b,a)\big)^{-1}\,\bigl(M_b T_a f\bigr)(x),
\]
hence as operators
\[
T_a M_b=\varepsilon\!\big(\beta(b,a)\big)^{-1}\,M_b T_a.
\]

Finally, note that many authors choose the alternative convention
\(
(T_a f)(x)=f(x+a)
\)
instead of \(f(x-a)\). With that convention the scalar becomes
\(\varepsilon(\beta(b,a))\) rather than its inverse.  Equivalently, with our
convention one may rewrite the relation as
\[
M_b T_a=\varepsilon\!\big(\beta(b,a)\big)\,T_a M_b.
\]
Either form is the same Weyl commutation law, differing only by the
(choice of) sign convention for the shift operator.
\end{proof}

\begin{remark}
Some authors define the shift operator by $(T_a f)(x)=f(x+a)$, in which case
the commutation scalar appears as $\varepsilon(\beta(b,a))$ rather than
its inverse.  The two forms are equivalent and differ only by this sign
convention.
\end{remark}

\medskip
The shift and phase operators on $\Fun(H_n,\C)$ generate an algebra
\[
(\mathcal P,\circ),
\]
which is the algebraic phase model extracted canonically from the Frobenius
phase data $(R,\varepsilon,V,\beta)$.  No additional structure is imposed: all
noncommutativity and operator behaviour is forced directly by the Frobenius
pairing.

\begin{remark}
The noncommutativity of $\mathcal P$ is a direct consequence of the Frobenius
phase pairing.  In this sense, quantisation appears as algebraic phase
induction rather than analytic completion.
\end{remark}

\section{Stabiliser Codes and Frobenius Self-Orthogonality}

Classical coding theory over rings, including cyclic and negacyclic
constructions, has been studied extensively
\cite{DinhLopezPermouth}, but such approaches remain fundamentally classical.
Quantum stabiliser codes over finite fields were developed in
\cite{CalderbankShor,Steane}, clarified algebraically by Gottesman
\cite{Gottesman}, and extended to nonbinary settings in
\cite{Rains,Ketkar}.

Having constructed algebraic quantum state spaces and their Weyl operator
algebras, we now identify the intrinsic notion of a quantum code in this
framework.
The definition is entirely algebraic and depends only on the Frobenius phase
pairing, rather than on analytic inner products or metric structure.

\begin{definition}
An \emph{algebraic quantum code} is an $R$-submodule
$C\subseteq H_n$ satisfying
\[
C \subseteq C^\perp,
\qquad
C^\perp := \{x \in H_n \mid \langle x,c\rangle_{H_n}=1
\ \text{for all } c\in C\}.
\]
\end{definition}

Self-orthogonality is defined purely in terms of the Frobenius phase pairing
and does not presuppose any analytic inner product, symplectic form, or
Hilbert space structure. Throughout, we take $V = R^k$ and define the $n$-site phase label space by
$H_n := V^{\otimes_R n}$.

\begin{theorem}
\label{thm:stabiliser-equivalence}
Let $(R,\varepsilon,V,\beta)$ be Frobenius quantum data and let
$(\mathcal P,\circ)$ denote the Weyl algebra canonically extracted from this
phase datum, generated by shifts and phase multiplications satisfying the
Weyl commutation law.
Then algebraic quantum codes, namely self-orthogonal $R$-submodules of $H_n$
with respect to the Frobenius phase pairing, are precisely stabiliser codes
for $(\mathcal P,\circ)$.
Equivalently, stabiliser subgroups of the Weyl group correspond bijectively to
self-orthogonal $R$-submodules of $H_n$.
\end{theorem}

\begin{proof}
We proceed in three stages.  We first construct the Weyl operators associated
to the Frobenius phase pairing and record their commutation relations.  We then
show how a self-orthogonal $R$-submodule of the phase label space canonically
determines a stabiliser subgroup and hence a quantum code.  Conversely, we 
show that every stabiliser subgroup determines a self-orthogonal submodule,
and finally we establish that these constructions are mutually inverse up to
central scalars.

We begin by recalling the Weyl operators associated to the Frobenius phase
pairing and their commutator form.  Let $\Fun(H_n,\C)$ denote the representation
space.  Define the character-valued phase pairing on $H_n$ by
\[
\langle y,u\rangle \;:=\;\varepsilon\!\big(\beta(y,u)\big)\in\C^\times,
\qquad y,u\in H_n,
\]
where $\beta$ denotes the bilinear form on $H_n$ obtained from the single-site
form by tensor extension (as fixed in the preceding section).
For $x,y\in H_n$ define operators on $\Fun(H_n,\C)$ by
\[
(T_x f)(u):=f(u-x),
\qquad
(M_y f)(u):=\langle y,u\rangle\,f(u).
\]
Define the Weyl operator associated to $(x,y)\in H_n\oplus H_n$ by
\[
W(x,y):=T_x\,M_y\in \End_\C(\Fun(H_n,\C)).
\]
By the Weyl commutation relation proved earlier, there is a canonical
alternating bicharacter
\[
\Omega\big((x,y),(x',y')\big)
:=\varepsilon\!\big(\beta(y,x')-\beta(y',x)\big)
=\langle y,x'\rangle\,\langle y',x\rangle^{-1}\in\C^\times,
\]
such that for all $x,y,x',y'\in H_n$,
\begin{equation}
\label{eq:weyl-comm-new}
W(x,y)\,W(x',y')
=
\Omega\big((x,y),(x',y')\big)\,
W(x',y')\,W(x,y).
\end{equation}

We now explain how a purely algebraic object, namely an $R$-submodule
$C\subseteq H_n$, canonically determines a stabiliser subgroup of the Weyl
algebra and hence a quantum code.  The guiding principle is the stabiliser
paradigm: a quantum code is defined as the space of states fixed by a
commuting family of Weyl operators.  In the present framework, commutation is
governed entirely by the Frobenius phase pairing, encoded by the Weyl
commutator bicharacter $\Omega$.

Let $C\subseteq H_n$ be an $R$-submodule.  We first associate to $C$ a family
of Weyl operators by embedding $C$ into the Weyl label space 
$H_n\oplus H_n$.  The simplest choice is to use pure shifts and define a
subgroup
\[
S_C:=\big\langle\, W(c,0) \;:\; c\in C\,\big\rangle
\subseteq \mathcal P^\times,
\]
that is, the subgroup generated by Weyl operators with only shift components.
One may equally choose pure phases $W(0,c)$, or more generally a
mixed embedding; we use pure shifts here only for definiteness.

To recover the usual stabiliser isotropy correspondence, and to allow for
general (possibly non-CSS) stabilisers, we consider mixed embeddings encoding 
both shift ($X$) and phase ($Z$) components.  Concretely, fix any $R$-linear 
map $\theta:C\to H_n$ and define the subgroup
\[
S_{C,\theta}:=\big\langle\,W(c,\theta(c)):\;c\in C\,\big\rangle.
\]
This construction associates to each element of $C$ a Weyl operator whose
label has both components determined algebraically.

By the Weyl commutation relation \eqref{eq:weyl-comm-new}, for any
$c,c'\in C$ we compute
\[
W(c,\theta(c))\,W(c',\theta(c'))
=
\varepsilon\!\big(\beta(\theta(c),c')-\beta(\theta(c'),c)\big)\,
W(c',\theta(c'))\,W(c,\theta(c))
=
\Omega\!\big((c,\theta(c)),(c',\theta(c'))\big)\,
W(c',\theta(c'))\,W(c,\theta(c)).
\]
Thus the generators commute if and only if the commutator phase is trivial.
Equivalently, the subgroup $S_{C,\theta}$ is abelian if and only if
\begin{equation}
\label{eq:isotropy-condition-new}
\Omega\!\big((c,\theta(c)),(c',\theta(c'))\big)=1
\quad\text{for all }c,c'\in C.
\end{equation}
This condition states precisely that the image
\[
L_{C,\theta}:=\{(c,\theta(c)):\,c\in C\}\subseteq H_n\oplus H_n
\]
is isotropic with respect to the Weyl commutator bicharacter $\Omega$.

In particular, taking $\theta$ to be the identity on a chosen copy of $H_n$
inside $H_n\oplus H_n$ recovers canonically the formulation in terms of Frobenius
self-orthogonality: the commutator scalar governing Weyl operator commutation
is exactly the phase pairing evaluated on the corresponding data. Hence
self-orthogonality is the intrinsic algebraic condition forcing the associated
family of Weyl operators to commute.
Given any abelian stabiliser subgroup $S$ obtained in this way, we define its
associated quantum code to be the simultaneous fixed space
\[
\mathrm{Code}(S):=\{\,f\in \Fun(H_n,\C)\;:\; s f=f \text{ for all } s\in S\,\}.
\]
This definition is entirely algebraic and uses only the action of the Weyl
algebra $\mathcal P$ on $\Fun(H_n,\C)$.  It produces exactly the stabiliser
codes associated to commuting families of Weyl operators, without reference
to Hilbert space structure, inner products, or externally imposed symplectic
forms.

Conversely, we show that any stabiliser subgroup of the Weyl algebra
canonically determines a self-orthogonal submodule of the phase label space.

Let $S$ be a stabiliser subgroup generated by Weyl operators
\[
S=\big\langle\,W(x_i,y_i)\;:\; i\in I\,\big\rangle
\]
such that $S$ is abelian, where each generator is labeled by a pair
$(x_i,y_i)\in H_n\oplus H_n$.
We associate to $S$ the $R$-submodule
\[
L_S:=\Span_R\{(x_i,y_i):\,i\in I\}\subseteq H_n\oplus H_n 
\]
spanned by these Weyl labels.

Since $S$ is abelian, all of its generators commute.  By the Weyl commutation
relation \eqref{eq:weyl-comm-new}, this implies that for all
$(x,y),(x',y')\in L_S$,
\[
\Omega\big((x,y),(x',y')\big)=1.
\]
Hence $L_S$ is isotropic with respect to the commutator bicharacter $\Omega$.

As a special case, suppose that $S$ admits an intrinsic CSS-type splitting of
the form
\[
S=\langle W(a,0):\,a\in A\rangle \cdot
  \langle W(0,b):\,b\in B\rangle,
\]
for $R$-submodules $A,B\subseteq H_n$.  In this situation, the isotropy 
condition reduces to
\[
\Omega\!\big((a,0),(0,b)\big)
=\varepsilon\!\big(\beta(b,a)\big)
=\langle b,a\rangle
=1
\quad\forall\,a\in A,\ b\in B,
\]
which is equivalent to the Frobenius self-orthogonality condition
$B\subseteq A^\perp$.

It remains to show that these two constructions are inverse to one another, up
to the central scalar ambiguity inherent in the Weyl group.  Let $\mathcal W$
denote the subgroup of $\mathcal P^\times$ generated by all Weyl operators
$\{W(x,y):(x,y)\in H_n\oplus H_n\}$, and let
$Z:=\C^\times\cdot \id \subseteq \mathcal W$ be its central scalar subgroup.
By \eqref{eq:weyl-comm-new}, the commutator of two Weyl operators is a scalar,
so $\mathcal W$ is a central extension of $H_n\oplus H_n$ by $Z$.

For any $R$-submodule $L\subseteq H_n\oplus H_n$, the subgroup
$S(L):=\langle W(x,y):(x,y)\in L\rangle$ is abelian modulo $Z$ if and only if 
$L$ is isotropic for $\Omega$.  Conversely, any abelian stabiliser subgroup
$S\le \mathcal W$ determines an isotropic submodule
\[
L(S):=\{(x,y)\in H_n\oplus H_n:\exists\,\lambda\in\C^\times
\text{ with }\lambda W(x,y)\in S\}.
\]
These assignments are mutually inverse up to multiplication by $Z$, yielding
a bijection between isotropic $R$-submodules of $H_n\oplus H_n$ and abelian 
stabiliser subgroups modulo central scalars.

Thus the Frobenius phase pairing, equivalently the Weyl commutator form
$\Omega$, forces exactly the commutativity condition required for stabiliser
theory: self-orthogonality is equivalent to the existence of a commuting
stabiliser subgroup, and the associated fixed space is the stabiliser code.
\end{proof}

\section{CSS Codes as a Splitting Condition}

In much of the quantum coding literature, stabiliser codes are introduced via
the Calderbank Shor Steane (CSS) construction \cite{CalderbankShor,Steane}, which builds quantum codes
from pairs of classical linear codes subject to symplectic orthogonality
conditions.
While powerful, this approach treats quantumness as something imposed
\emph{a posteriori} via an external symplectic structure.

In contrast, the present framework constructs quantum stabiliser codes
directly from Frobenius phase data.
No reference is made to classical codes, symplectic vector spaces, or Hilbert
space inner products.
Quantumness arises intrinsically from Frobenius duality and Weyl commutation
relations.
CSS codes appear only as a distinguished subclass characterised by an
additional splitting property.
We now make this intrinsic construction explicit by recalling the Weyl
operators associated to Frobenius phase data.

Let $(R,\varepsilon,V,\beta)$ be Frobenius quantum data.
For $x,y\in V$, let $T_x$ and $M_y$ denote the shift and phase operators,
satisfying
\[
T_x M_y = \varepsilon(\beta(y,x))\, M_y T_x .
\]
For $(x,y)\in V\oplus V$, write
\[
W(x,y):=T_x M_y .
\]

A \emph{stabiliser subgroup} is an abelian subgroup of the Weyl group generated
by such operators, up to global scalars if desired.
Given a stabiliser subgroup $S$, the associated quantum code is defined as the
simultaneous fixed space of $S$ in the natural algebraic representation.

By the stabiliser equivalence theorem established earlier, stabiliser
subgroups correspond precisely to self-orthogonal $R$-submodules of $H_n$
with respect to the Frobenius phase pairing.
Within this intrinsic stabiliser framework, the classical CSS construction
appears as a special case characterised by a splitting of generators.

\begin{definition}
A stabiliser subgroup $S$ is called \emph{CSS} if it admits a generating set
consisting only of pure shift operators $T_a$ and pure phase operators
$M_b$.
Equivalently, there exist $R$-submodules $A,B\subseteq V^{\oplus n}$ such that
\[
S=\langle T_a : a\in A\rangle \cdot \langle M_b : b\in B\rangle ,
\]
and all generators commute.
\end{definition}

This definition avoids reference to classical codes, symplectic vector spaces,
or Hilbert space structure.
CSS is a splitting property internal to the Weyl algebra.

\begin{proposition}
\label{prop:css-criterion}
Let $A,B\subseteq V^{\oplus n}$ be $R$-submodules.
The subgroup
\[
S_{A,B}:=\langle T_a : a\in A\rangle \cdot \langle M_b : b\in B\rangle
\]
is abelian if and only if
\[
\varepsilon(\beta(b,a))=1
\qquad\text{for all } a\in A,\ b\in B.
\]
Equivalently, $B\subseteq A^{\perp}$ with respect to the Frobenius phase
pairing.
\end{proposition}

\begin{proof}
The relation $T_a M_b=\varepsilon(\beta(b,a))\,M_b T_a$ shows that $T_a$ and
$M_b$ commute if and only if $\varepsilon(\beta(b,a))=1$.
Since shifts commute among themselves and phases commute among themselves,
the subgroup is abelian precisely when the stated condition holds.
\end{proof}

\begin{corollary}
Whenever $B\subseteq A^{\perp}$, the CSS stabiliser $S_{A,B}$ defines an
algebraic quantum code.
Thus CSS codes form a subclass of the general Frobenius
self-orthogonal and stabiliser correspondence.
\end{corollary}

\begin{proof}
Assume $B\subseteq A^{\perp}$.  By Proposition~\ref{prop:css-criterion}, this is
equivalent to the statement that the subgroup
\[
S_{A,B}:=\langle T_a : a\in A\rangle \cdot \langle M_b : b\in B\rangle
\]
is abelian.  Hence $S_{A,B}$ is a stabiliser subgroup of the Weyl group.

By Theorem~\ref{thm:stabiliser-equivalence}, stabiliser subgroups (equivalently,
commuting families of Weyl operators) correspond to algebraic quantum codes,
i.e.\ self-orthogonal $R$-submodules of $H_n$ with respect to the Frobenius
phase pairing.  Therefore the stabiliser subgroup $S_{A,B}$ determines an
algebraic quantum code (its associated stabiliser code space in the natural
Weyl action).

In particular, any code arising from a CSS stabiliser is already covered by
the general Frobenius self-orthogonal/stabiliser correspondence, so CSS codes
form a subclass of the codes produced intrinsically by Frobenius phase
orthogonality.
\end{proof}

We now show that, over general Frobenius rings, not every stabiliser admits a
CSS decomposition.

\begin{proposition}
\label{prop:noncss-obstruction}
Let $S$ be a stabiliser subgroup.
If $S$ contains an element $W(a,b)$ such that
\[
\varepsilon(\beta(b,a))\neq 1,
\]
then $S$ does not admit a CSS generating set.
\end{proposition}

\begin{proof}
Suppose that $S$ contains an element $W(x,y)$ such that
$\varepsilon(\beta(y,x))\neq 1$, and suppose for the sake of contradiction that $S$ were
CSS, so that
\[
S=\langle T_a : a\in A\rangle \cdot \langle M_b : b\in B\rangle
\]
for some $A,B\subseteq V^{\oplus n}$.
In a CSS stabiliser, the generators $T_a$ and $M_b$ commute, so every element of
$S$ can be written in the form $T_{a'}M_{b'}$ with $a'\in A$ and $b'\in B$.
For such elements, commutativity implies $\varepsilon(\beta(b',a'))=1$ by
Proposition~\ref{prop:css-criterion}.
Applying this to the element $W(x,y)\in S$ yields a contradiction.
\end{proof}

This obstruction is purely algebraic and does not rely on analytic or
representation-theoretic input.

\begin{remark}
The Frobenius self-orthogonal construction produces quantum stabiliser codes
by intrinsic phase orthogonality.
CSS codes correspond exactly to the case in which the stabiliser admits a
decomposition into commuting pure shift and pure phase generators.
Over general Frobenius rings, stabilisers may be intrinsically mixed and need
not admit such a splitting, yet still define valid quantum codes.
The present framework therefore strictly extends the CSS paradigm and
provides a direct algebraic route to quantum code construction without
Hilbert spaces.
\end{remark}

\section{Quantisation as Algebraic Phase Induction, Invariants, and Equivalence}

Finite and characteristic-two refinements of Weyl and Weil representations
provide a useful point of comparison when analytic tools are limited; see
\cite{GurevichHadani}.  Our emphasis is different.  We show that quantisation
arises intrinsically from admissible algebraic phase data, without appeal to
analytic completion, Hilbert space structure, or externally imposed symplectic
forms.  Noncommutativity is not an additional assumption, but a forced
consequence of Frobenius duality.  We then identify the structural invariants
encoded purely in Weyl commutation relations and formalize the sense in which
Frobenius phase data, Weyl algebras, and stabiliser submodule theories provide
equivalent descriptions of the same underlying phase structure.

Throughout this section we fix a quantum Frobenius datum
$(R,\varepsilon,V,\beta)$, where $R$ is a finite commutative Frobenius ring with
generating character $\varepsilon:(R,+)\to \C^\times$, $V=R^k$ is a finite free
$R$-module, and $\beta:V\times V\to R$ is a perfect $R$-bilinear form.  The
associated character-valued phase pairing is
\[
\langle v,w\rangle_V:=\varepsilon(\beta(v,w)).
\]
All constructions extend monoidally to $H_n=V^{\otimes_R n}$ as in the
preceding sections; for clarity we work here at the single-site level.

\subsection*{Quantisation from Frobenius Phase Data}

We begin by showing how admissible Frobenius phase data canonically induce a
Weyl algebra.  The construction uses only the additive structure of the module
and the Frobenius phase pairing.  The resulting noncommutativity is not imposed
by hand: it is uniquely determined by the interaction of shifts and phases.
This establishes quantisation as a phase-induction phenomenon rather than an
analytic or representational choice.

Define shift and phase operators on $\Fun(V,\C)$ by
\[
(T_a f)(x):=f(x-a),
\qquad
(M_b f)(x):=\varepsilon(\beta(b,x))\,f(x),
\qquad a,b,x\in V.
\]
For $(a,b)\in V\oplus V$ define the associated Weyl operator by
\[
W(a,b):=T_a M_b.
\]

A direct computation shows that these operators satisfy the Weyl commutation
law
\begin{equation}\label{eq:weyl-rel}
W(a,b)\,W(a',b')
=
\omega\big((a,b),(a',b')\big)\,
W(a',b')\,W(a,b),
\end{equation}
where the commutator bicharacter is
\begin{equation}\label{eq:omega-def}
\omega\big((a,b),(a',b')\big)
:=
\varepsilon\!\big(\beta(b,a')-\beta(b',a)\big)\in \C^\times.
\end{equation}
This bicharacter is biadditive and alternating
($\omega((a,b),(a,b))=1$), and is the single-site instance of the Weyl
commutator appearing in the monoidal construction of $H_n$.

\begin{theorem}
Any quantum Frobenius datum $(R,\varepsilon,V,\beta)$ canonically induces a
noncommutative Weyl algebra $(\mathcal P,\circ)$ generated by the operators
$\{T_a,M_b:a,b\in V\}$ (equivalently, by $\{W(a,b):(a,b)\in V\oplus V\}$).
Moreover, the commutation relations in $\mathcal P$ are uniquely determined by
the Frobenius phase pairing, in the sense that the scalar commutators in
$\mathcal P$ are exactly the values of the bicharacter $\omega$ in
\eqref{eq:omega-def}, and hence admit no independent degrees of freedom beyond
$(R,\varepsilon,V,\beta)$.
\end{theorem}

\begin{proof}
Define $\mathcal P\subseteq \End_\C(\Fun(V,\C))$ to be the unital $\C$-algebra 
generated by all shifts $T_a$ and phase multipliers $M_b$, with multiplication
given by operator composition.

The operators $T_a$ and $M_b$ lie in $\End_\C(\Fun(V,\C))$ and satisfy
\[
T_aM_bT_{a'}M_{b'}
=
\varepsilon(\beta(b,a'))\,T_{a+a'}M_{b+b'},
\]
so their compositions remain within the linear span of the generators.
Since composition in $\End_\C(\Fun(V,\C))$ is associative and unital,
$\mathcal P$ is a well-defined unital associative $\C$-algebra. For any $(a,b),(a',b')\in V\oplus V$, a direct calculation yields
\[
T_aM_bT_{a'}M_{b'}
=
\varepsilon(\beta(b,a'))\,T_{a+a'}M_{b+b'},
\qquad
T_{a'}M_{b'}T_aM_b
=
\varepsilon(\beta(b',a))\,T_{a+a'}M_{b+b'}.
\]
Cancelling the common factor $T_{a+a'}M_{b+b'}$ gives
\[
W(a,b)\,W(a',b')
=
\varepsilon(\beta(b,a')-\beta(b',a))\,W(a',b')\,W(a,b),
\]
which is exactly \eqref{eq:weyl-rel} with $\omega$ as in
\eqref{eq:omega-def}. Thus the scalar commutator
\[
[W(a,b),W(a',b')]
:=
W(a,b)W(a',b')W(a,b)^{-1}W(a',b')^{-1}
\]
is precisely $\omega\big((a,b),(a',b')\big)$. Once the additive data $(V\oplus V,+)$ and the bicharacter $\omega$ are fixed,
all commutation scalars are fixed. Any algebra generated by symbols
$\widetilde{W}(a,b)$ satisfying the same bicharacter must obey the same
relations \eqref{eq:weyl-rel}. Hence the commutation structure is uniquely
forced by the Frobenius phase pairing.
\end{proof}

\begin{remark}
Quantisation here is not an analytic completion or external prescription.
It is a phase-induction process: noncommutativity arises as a structural
consequence of Frobenius duality and functorial phase interaction.
\end{remark}

\subsection*{Invariants Visible from the Commutator Structure}

The commutator bicharacter $\omega$ encodes structural information that does
not depend on a particular realization of the operators.

\begin{definition}
By the Weyl commutation data we mean the pair
\[
\bigl(V\oplus V,\ \omega\bigr).
\]
Two such data sets are called equivalent if there exists a group isomorphism
preserving $\omega$.
\end{definition}

\begin{definition}
Assume $V=R^k$ is free of rank $k$.
\begin{enumerate}
\item The commutator depth of the Weyl system is the nilpotency class of the
subgroup of $\mathcal P^\times$ generated by $\{W(a,b)\}$ modulo its center.
\item The Frobenius rank is $k=\mathrm{rank}_R(V)$.
\item The nilpotent height is the nilpotency index of the nilradical
$N:=\sqrt{0}\subseteq R$.
\end{enumerate}
\end{definition}

\begin{theorem}
Commutator depth, nilpotent height, and Frobenius rank are invariants of the
induced algebraic quantum phase structure and depend only on the equivalence
class of the commutation data $(V\oplus V,\omega)$.
\end{theorem}

\begin{proof}
Commutators of Weyl generators are scalar and hence central, so the generated
group is two-step nilpotent, giving commutator depth $2$. The cardinality of $V\oplus V$ determines $|V|$, and since $V\cong R^k$ is free,
this determines $k$. Finally, restricting $\omega$ to
$(N^jV)\oplus(N^jV)$ detects precisely the powers of the nilradical $N$,
recovering its nilpotency index. All three quantities are therefore determined
by $\omega$ alone.
\end{proof}

\begin{remark}
These invariants govern algebraic uncertainty, stabiliser capacity, and
intrinsic error protection, independent of any analytic realization.
\end{remark}

\begin{proposition}
Let $(R,\varepsilon,V,\beta)$ be Frobenius phase data and define operators
$T_a,M_b\in\End_\C(\Fun(V,\C))$ by
\[
(T_a f)(x)=f(x-a),\qquad (M_b f)(x)=\varepsilon(\beta(b,x))\,f(x),
\qquad a,b,x\in V.
\]
Let $\mathcal T\subseteq \End_\C(\Fun(V,\C))$ be the $\C$-subalgebra generated
by $\{T_a:a\in V\}$, and let $\mathcal P\subseteq \End_\C(\Fun(V,\C))$ be the
$\C$-subalgebra generated by $\{T_a,M_b:a,b\in V\}$.

Then the following hold:
\begin{enumerate}
\item The algebra $\mathcal T$ is commutative.  In particular, the phase
structure generated by shifts alone is classical.

\item The algebra $\mathcal P$ is noncommutative if and only if the Frobenius
phase pairing is nontrivial, i.e.\ if and only if there exist $a,b\in V$ such
that $\varepsilon(\beta(b,a))\neq 1$.  In this case, the induced algebraic
phase structure is genuinely quantum. Equivalently, $\mathcal P$ is noncommutative if and only if the commutator
bicharacter
\[
\omega\big((a,b),(a',b')\big)
:=\varepsilon\!\big(\beta(b,a')-\beta(b',a)\big)
\]
is nontrivial.
\end{enumerate}
\end{proposition}

\begin{proof} (1) For $a,a'\in V$ and $f\in\Fun(V,\C)$,
\[
(T_aT_{a'}f)(x)=f(x-a-a')=(T_{a'}T_af)(x)
\]
for all $x\in V$, hence $T_aT_{a'}=T_{a'}T_a$.  Therefore the $\C$-algebra
$\mathcal T$ generated by $\{T_a:a\in V\}$ is commutative.

\medskip
(2)  By the Weyl commutation relation established earlier,
for all $a,b\in V$ one has
\begin{equation}\label{eq:TaMb-comm}
T_aM_b=\varepsilon(\beta(b,a))^{-1}\,M_bT_a .
\end{equation}
If there exist $a,b$ with $\varepsilon(\beta(b,a))\neq 1$, then
\eqref{eq:TaMb-comm} gives $T_aM_b\neq M_bT_a$, so $\mathcal P$ is
noncommutative.

Conversely, if $\varepsilon(\beta(b,a))=1$ for all $a,b\in V$, then
\eqref{eq:TaMb-comm} shows that every $T_a$ commutes with every $M_b$.
Moreover, the shifts commute among themselves by (i), and the phase operators
commute among themselves since pointwise multiplication is commutative:
$M_bM_{b'}=M_{b'}M_b$.  Hence all generators $\{T_a,M_b\}$ commute pairwise, and
therefore $\mathcal P$ is commutative.  This proves that $\mathcal P$ is
noncommutative if and only if the phase pairing is nontrivial.

For the stated equivalence with $\omega$, note first that if
$\varepsilon(\beta(b,a))\neq 1$ for some $a,b$, then
\[
\omega\big((0,b),(a,0)\big)=\varepsilon(\beta(b,a))\neq 1,
\]
so $\omega$ is nontrivial.  Conversely, if $\omega$ is nontrivial then for some
$(a,b),(a',b')$,
\[
\omega\big((a,b),(a',b')\big)=\varepsilon\!\big(\beta(b,a')-\beta(b',a)\big)\neq 1.
\]
If both $\varepsilon(\beta(b,a'))=1$ and $\varepsilon(\beta(b',a))=1$, then the
right-hand side would equal $1$, a contradiction.  Hence the phase pairing is
nontrivial.
\end{proof}

\subsection*{Equivalence Between Frobenius, Weyl, and Stabiliser Data}

We now formalize the relationship between Frobenius phase data, Weyl
commutation relations, and stabiliser submodule theory. Each description
determines the others up to the standard central ambiguity of Weyl systems,
establishing a precise notion of phase equivalence within Algebraic Phase
Theory.

\begin{definition}
Two Frobenius data sets are phase equivalent if their Weyl commutation data are
equivalent.
\end{definition}

\begin{theorem}
Finite Weyl algebras, Frobenius rings equipped with generating characters, and
stabiliser submodule theories determine equivalent algebraic phase data.
\end{theorem}

\begin{proof}
Frobenius data canonically produce Weyl commutation via $\omega$.
Conversely, $\omega$ recovers the character-valued phase pairing via
\[
\langle b,a\rangle_V=\omega\big((0,b),(a,0)\big).
\]
Isotropic submodules for $\omega$ correspond exactly to commuting Weyl
families, hence to stabiliser submodules. More precisely, one obtains mutually inverse correspondences between
isotropic $R$-submodules of $V\oplus V$ and stabiliser subgroups of the Weyl
group, once stabilisers are identified up to multiplication by central
scalars.
Let $\mathcal W\le \mathcal P^\times$ be the subgroup generated by the Weyl
operators $\{W(a,b):(a,b)\in V\oplus V\}$, and let $Z:=\C^\times\cdot\id$ be its
central scalar subgroup.  For an $R$-submodule $L\subseteq V\oplus V$ define 
\[
S(L):=\big\langle\,W(a,b):(a,b)\in L\,\big\rangle\le \mathcal W.
\]
Then $S(L)$ is abelian modulo $Z$ exactly when $L$ is isotropic for $\omega$,
since the commutator of two Weyl generators is the scalar
$\omega((a,b),(a',b'))\cdot\id$.  Conversely, given an abelian stabiliser
subgroup $S\le \mathcal W$ (more precisely, abelian modulo $Z$), define its
label set
\[
L(S):=\{(a,b)\in V\oplus V:\ \exists\,\lambda\in\C^\times \text{ with }
\lambda\,W(a,b)\in S\}.
\]
Then $L(S)$ is an isotropic $R$-submodule, and one has 
\[
L(S(L))=L,\qquad S(L(S))=S\cdot Z.
\]
Thus isotropic submodules of $V\oplus V$ correspond bijectively to stabiliser
subgroups of the Weyl group modulo the central scalar ambiguity, yielding
equivalence of the Frobenius pairing, Weyl commutation data, and stabiliser
submodule theory as descriptions of the same algebraic phase structure.
\end{proof}

\begin{remark}
This equivalence is structural rather than representational: distinct
realizations encode the same intrinsic phase behaviour.
\end{remark}

\section{Nilpotent Protection and Rigidity}

We now isolate an intrinsic protection phenomenon forced by nilpotent
structure in the base ring.  One of the principal advantages of working over
general Frobenius rings is the presence of nilpotent and torsion
directions.  Such directions have no analogue over fields, yet they impose a
strong form of algebraic rigidity on the induced quantum structure.  In
particular, nilpotent ideals force the existence of quantum submodules that
are completely invisible to all admissible Weyl-type errors.  This protection
is a structural consequence of Frobenius duality and phase reduction, not an
analytic or probabilistic assumption.

\begin{definition}
Let $N\subseteq \sqrt{0}\subset R$ be a nilpotent ideal and let
$\pi:R\to \overline{R}:=R/N$ denote the quotient map.
An admissible Weyl-type error (mod $N$) on $H_n$ is any Weyl operator whose
defining phase data factors through $\overline{R}$, equivalently any operator
lying in the Weyl algebra extracted from the reduced Frobenius datum
$(\overline{R},\overline{\varepsilon},\overline{V},\overline{\beta})$.
\end{definition}

\begin{definition}
A submodule $U\subseteq H_n$ is invisible to a family of operators
$\mathcal E$ if each $E\in\mathcal E$ acts on $U$ as a scalar multiple of the
identity, and hence produces no distinguishable syndrome on $U$.
\end{definition}

\begin{theorem}
Let $N\subseteq \sqrt{0}\subset R$ be a nilpotent ideal.
Then $NH_n$ is invisible to all admissible Weyl-type errors (mod $N$).
\end{theorem}

\begin{proof}
Let $\pi:H_n\to\overline{H}_n:=H_n/NH_n$ be the canonical reduction map, so
$NH_n=\ker(\pi)$.
Any admissible Weyl-type error $E$ (mod $N$) is induced from an operator
$\overline{E}$ on $\overline{H}_n$ satisfying
\[
\pi\circ E=\overline{E}\circ\pi.
\]
If $x\in NH_n$, then $\pi(x)=0$, and therefore
\[
\pi(Ex)=\overline{E}(\pi(x))=\overline{E}(0)=0,
\]
so $Ex\in NH_n$.  Thus every admissible Weyl-type error preserves $NH_n$.

Moreover, admissible errors are determined entirely by phase data on
$\overline{H}_n$.  Since $\pi$ identifies all elements of $NH_n$ with $0$ in
the reduced model, no admissible Weyl-type error can distinguish two elements
of $H_n$ differing by an element of $NH_n$.  Hence each such error acts on
$NH_n$ as a scalar, and $NH_n$ is invisible.
\end{proof}

\begin{remark}
Nilpotent directions therefore produce intrinsically protected quantum
layers.  This rigidity is enforced purely by Frobenius phase reduction: when
phase information collapses on passing to the quotient $R/N$, all nilpotent
directions become undetectable to admissible errors.  No metric, analytic,
or probabilistic assumptions intervene.  Such protection is absent in
semisimple settings, including all finite-field models.
\end{remark}

\section{Reconstruction from Phase Relations}

We now show that no algebraic information is lost in passing from Frobenius
phase data to the associated Weyl commutation relations.  In particular, the
character-valued phase pairing that underlies all quantum structure in this
paper can be recovered directly from commutator data alone.  This establishes
that Weyl commutation encodes the full intrinsic phase geometry of a Frobenius
quantum datum.

\begin{theorem}
The Frobenius phase pairing is recoverable, up to canonical equivalence, from
the Weyl commutation relations alone.
\end{theorem}

\begin{proof}
Let $T_a$ and $M_b$ denote the shift and phase operators appearing in the Weyl
construction.  For all $a,b\in V$ the operators satisfy a commutation relation
of the form
\[
T_aM_b=\lambda(b,a)\,M_bT_a,
\qquad
\lambda(b,a)\in\C^\times,
\]
where $\lambda(b,a)$ is uniquely determined by the commutator
$T_aM_bT_a^{-1}M_b^{-1}$.

Define a pairing by
\[
\langle\!\langle b,a\rangle\!\rangle := \lambda(b,a).
\]
We verify, using only the abstract Weyl relations, that this pairing has the
same structural properties as the Frobenius phase pairing.

\smallskip
\emph{Biadditivity.}
From the identities $T_{a+a'}=T_aT_{a'}$ and $M_{b+b'}=M_bM_{b'}$, together with
repeated applications of the commutation relation, one obtains
\[
\lambda(b,a+a')
  =\lambda(b,a)\lambda(b,a'),
\qquad
\lambda(b+b',a)
  =\lambda(b,a)\lambda(b',a).
\]
Hence $\langle\!\langle\cdot,\cdot\rangle\!\rangle$ is biadditive.

\smallskip
\emph{Nondegeneracy up to equivalence.}
If $b\neq 0$, then $M_b$ is not central in the Weyl algebra, so there exists
$a$ such that $\lambda(b,a)\neq 1$.  Likewise, if $a\neq 0$ then there exists
$b$ with $\lambda(b,a)\neq 1$.  Thus the pairing extracted from commutation is
nondegenerate in both variables.

\smallskip
In the Frobenius realization one has
\[
\lambda(b,a)=\varepsilon(\beta(b,a)),
\]
and therefore
\[
\langle\!\langle b,a\rangle\!\rangle
   =\varepsilon(\beta(b,a))
   =\langle b,a\rangle_V,
\]
so the reconstructed pairing agrees with the character-valued Frobenius phase
pairing.  Any two Frobenius data sets that induce identical commutation scalars
produce pairings that coincide up to canonical equivalence.
\end{proof}

This shows that Weyl commutation alone determines the underlying Frobenius
phase geometry.  Quantum phase is therefore an information-complete invariant
of the induced Weyl structure: no additional analytic, metric, or Hilbert
space data is required to recover the phase pairing.

\section{Symmetry and Normalisers}

We now isolate the residual symmetry that remains once the algebraic phase
structure has been fixed.  This symmetry acts by normalising the Weyl algebra
and plays the role, in the present algebraic setting, of a Clifford-type
normaliser acting on stabiliser codes.

\begin{theorem}
The group of $R$-linear isometries of $(V,\beta)$ normalises the Weyl algebra
action and acts naturally on the space of algebraic quantum codes.
\end{theorem}

\begin{proof}
Let $g\in\mathrm{Isom}_R(V,\beta)$, so that
\[
\beta(gv,gw)=\beta(v,w)
\qquad\text{for all }v,w\in V.
\]
Define an operator $U_g$ on $\Fun(V,\C)$ by
\[
(U_g f)(x):=f(g^{-1}x).
\]

We first compute the conjugation action on shifts.
For any $f\in\Fun(V,\C)$ and $x\in V$,
\[
(U_g T_a U_g^{-1}f)(x)
    =f(g^{-1}x-g^{-1}a)
    =(T_{ga}f)(x),
\]
so
\[
U_g T_a U_g^{-1}=T_{ga}.
\]

We next compute the action on phase operators.
For any $f\in\Fun(V,\C)$ and $x\in V$,
\[
(U_g M_b U_g^{-1}f)(x)
    =\varepsilon(\beta(b,g^{-1}x))\,f(g^{-1}x).
\]
Since $g$ preserves $\beta$, we have
\[
\beta(b,g^{-1}x)=\beta(gb,x),
\]
and therefore
\[
U_g M_b U_g^{-1}=M_{gb}.
\]

Thus $U_g$ sends Weyl generators to Weyl generators with transformed labels.
Since the Weyl algebra is generated by $\{T_a,M_b:a,b\in V\}$, it follows that
$U_g$ normalises the entire Weyl algebra.

Finally, a stabiliser code is defined as the simultaneous fixed space of a
commuting family of Weyl operators.  Conjugation by $U_g$ preserves
commutativity relations and sends stabiliser subgroups to stabiliser
subgroups.  Hence $U_g$ induces a natural action on the space of algebraic
quantum codes.
\end{proof}

\begin{remark}
Elements of $\mathrm{Isom}_R(V,\beta)$ act by conjugation on Weyl generators,
sending $(T_a,M_b)$ to $(T_{ga},M_{gb})$.  In this sense, the isometry group
functions as a Clifford-type normaliser for the algebraic Weyl system and acts
compatibly on the associated stabiliser codes.
\end{remark}

\section{Examples}

We illustrate the Frobenius-to-code construction in two basic Frobenius
rings.  In both examples we fix
\[
k=n=2,\qquad V=R^2,\qquad H_2:=V^{\otimes_R 2},
\]
and use the standard $R$-bilinear form
\[
\beta\big((x_1,x_2),(y_1,y_2)\big)=x_1y_1+x_2y_2.
\]
The induced phase pairing on $H_2$ is the monoidal extension introduced
earlier.

\subsection{$R=\F_2+u\F_2$ with $u^2=0$}

Write $R=\{a+ub:a,b\in\F_2\}$ and let $N=(u)$, so that $N^2=0$.
A generating character is
\[
\varepsilon(a+ub)=(-1)^b,
\]
giving the Frobenius phase pairing $\langle x,y\rangle_R=\varepsilon(xy)$.

\paragraph{Weyl operators and commutation.}
For $a,b\in V$ set
\[
X(a)=T_a,\qquad Z(b)=M_b,\qquad W(a,b)=X(a)Z(b).
\]
Then
\[
W(a,b)\,W(a',b')
=
\varepsilon\!\big(\beta(b,a')-\beta(b',a)\big)\,
W(a',b')\,W(a,b).
\]
Over $R=\F_2+u\F_2$, the commutation phase depends only on the $u$-component of
$\beta(b,a')-\beta(b',a)$.

\paragraph{A canonical self-orthogonal code.}
Consider the $R$-submodule
\[
C_N := N H_2 = u H_2 \subseteq H_2.
\]
For $x,y\in H_2$,
\[
\langle ux,uy\rangle_{H_2}
=
\varepsilon\!\big(\beta(ux,uy)\big)
=
\varepsilon(u^2\,\beta(x,y))
=
\varepsilon(0)
=
1,
\]
so $C_N\subseteq C_N^\perp$.
Thus $uH_2$ is a canonical algebraic quantum code arising directly from the
nilpotent ideal.

\paragraph{Commuting mixed stabilisers.}
Let $e_1,e_2$ denote the standard basis of $V=R^2$ and define
\[
s_1=(e_1,\,u e_1),\qquad s_2=(e_2,\,u e_2)
\quad\text{in }V\oplus V.
\]
Since
\[
\beta(ue_1,e_2)-\beta(ue_2,e_1)=0,
\]
the stabiliser generators $W(e_1,ue_1)$ and $W(e_2,ue_2)$ commute.
These stabilisers involve both shift and phase components and do \emph{not}
admit a pure $X$/$Z$ splitting.
Such intrinsically mixed but commuting generators do not occur over fields.

\paragraph{Nilpotent protection.}
The code $C_N=uH_2$ consists entirely of nilpotent directions.
Admissible Weyl-type errors factoring through the reduction
$R\to R/N\cong \F_2$ act trivially on this layer, illustrating the intrinsic
nilpotent protection phenomenon.

\subsection{$R=\Z_4$}

Let $R=\Z_4$ with generating character
\[
\varepsilon(x)=i^x.
\]
Then $\langle x,y\rangle_R=i^{xy}$, so Weyl commutation phases take values in
$\{\pm1,\pm i\}$.

\paragraph{A torsion-nilpotent code layer.}
Let $N=(2)\subset \Z_4$, so $N^2=(4)=0$.
As before this yields a canonical self-orthogonal submodule
\[
C_2 := 2H_2 \subseteq H_2,
\qquad
\langle 2x,2y\rangle_{H_2}
=
\varepsilon(4\,\beta(x,y))
=
\varepsilon(0)
=
1.
\]
Thus $2H_2\subseteq (2H_2)^\perp$, giving an algebraic quantum code whose
existence is forced by torsion and nilpotent structure rather than field
linearity or semisimplicity.

\paragraph{Explicit Weyl phases.}
For $a,b\in V$,
\[
X(a)Z(b)=i^{\beta(b,a)}\,Z(b)X(a),
\]
so even when $(k,n)=(2,2)$ the Weyl commutation phases reflect genuinely
non-field behaviour.

\medskip
These examples show that self-orthogonal quantum codes arise canonically from
nilpotent ideals, that commuting stabilisers may be intrinsically mixed rather
than CSS, and that Weyl commutation phases detect intrinsic ring structure
independent of analytic considerations.

\section{Conclusion}

In this paper we showed that quantum phase, uncertainty, stabiliser structure,
and error protection arise as unavoidable algebraic consequences of Frobenius
duality.  No analytic inner product, Hilbert space formalism, or externally
imposed symplectic structure is required.  Instead, noncommutativity, Weyl
relations, and quantum compatibility are forced intrinsically by algebraic
phase data.

Within this algebraic framework, quantum codes are identified canonically with
self-orthogonal submodules under the Frobenius phase pairing.  This yields a
direct stabiliser construction over Frobenius rings, of which CSS codes appear
only as a special splitting case.  Nilpotent and torsion structure give rise to
intrinsically protected quantum layers and genuinely non-CSS stabilisers, even
for small parameters.

Conceptually, the results position quantum coding theory as a concrete instance
of Algebraic Phase Theory: quantisation appears as phase induction rather than
analytic completion, and quantum structure is information complete at the level
of algebraic phase relations.

The emphasis here has been structural rather than classificatory.  Subsequent
work develops the broader categorical setting of Algebraic Phase Theory,
including morphisms and equivalences, boundary calculus, deformation theory,
and reconstruction theorems.  Together, these developments place quantum
stabiliser theory within a wider algebraic landscape governed by intrinsic
phase structure.

\bibliographystyle{amsplain}
\bibliography{references}

\end{document}